\UseRawInputEncoding 
\documentclass[reqno,11pt,a4paper]{amsart}

\usepackage{tikz}
\usetikzlibrary{matrix,arrows,calc,patterns,decorations}

\usepackage[mathscr]{eucal}
\usepackage{graphics,epic}
\usepackage{amsfonts}
\usepackage{amscd}
\usepackage{latexsym}
\usepackage{amsmath,amssymb, amsthm, stmaryrd, bm, bbm}
\usepackage[all,2cell]{xy}
\usepackage{mathrsfs}
\usepackage{url, hyperref}

\hypersetup{
    colorlinks,
    linkcolor={red!50!black},
    citecolor={blue!50!black},
    urlcolor={blue!80!black}
}

\newtheorem{thm}{Theorem}[section]
\newtheorem*{thm*}{Theorem}

\newtheorem{lem}[thm]{Lemma}
\newtheorem{prop}[thm]{Proposition}
\newtheorem{cor}[thm]{Corollary}

\newtheorem*{conj*}{Conjecture}

\newtheorem*{question*}{Question}
\theoremstyle{remark}
\newtheorem{rem}[thm]{Remark}
\newtheorem{ex}[thm]{Example}
\theoremstyle{definition}
\newtheorem{defn}[thm]{Definition}

\newcommand{\prdim}{\opname{pr.dim}\nolimits}

\newcommand{\opname}[1]{\operatorname{\mathsf{#1}}}

\newcommand{\rk}{\opname{rk}\nolimits}
\newcommand{\thick}{\opname{thick}\nolimits}

\DeclareMathOperator{\Coh}{\mathsf{coh}}

\newcommand{\Z}{\mathbb{Z}}

\newcommand{\C}{\mathbb{C}}

\renewcommand{\P}{\mathbb{P}}

\newcommand{\Hom}{\opname{Hom}}
\newcommand{\End}{\opname{End}}

\newcommand{\Ext}{\opname{Ext}}

\newcommand{\ca}{{\mathcal A}}

\newcommand{\cd}{{\mathcal D}}

\newcommand{\cp}{{\mathcal P}}

\newcommand{\ct}{{\mathcal T}}
\newcommand{\cu}{{\mathcal U}}

\numberwithin{equation}{section}

\begin{document}

\title[A finite dimensional algebra with a phantom]{A finite dimensional algebra with a phantom \\ (a corollary of an example by J.~Krah)}

\author{Martin Kalck}
\email{martin.kalck@math.uni-freiburg.de}

\thanks{The observations in this note can be seen as an appendix to Krah's paper \cite{Krah}. I thank Henning Krause for encouraging me to submit this note to the arXiv. I was supported by the DFG grant KA4049/2-1.}

\begin{abstract}
We observe that there exists an associative finite dimensional $\C$-algebra $A$ of finite global dimension, such that the bounded derived category $D^b(A)$ of finite dimensional $A$-modules admits an admissible subcategory $\cp$ with vanishing Grothendieck group $K_0(\cp)$. In other words, $\cp \subseteq D^b(A)$ is a phantom. 

Using tilting theory, this follows directly from a very recent example of a phantom for a smooth rational surface due to Krah.

By work of Aihara \& Iyama, this also leads to new examples of presilting objects that cannot be completed to silting objects.
\end{abstract}

\maketitle

\section{Phantoms}
We first recall the main result of Krah \cite[Theorem 1.1.]{Krah}.

\begin{thm}\label{T:Krah}
Let $X$ be the blow-up of $\P^2_\C$ in $10$ closed points in general position.

Then there exists an exceptional sequence $(E_1, \ldots, E_{13})$ in $D^b(\Coh X)$ which is \emph{not} full.
\end{thm}

Since $K_0(X)\cong \Z^{13}$, he obtains the following, cf.~\cite[Corollary 5.1.]{Krah}.

\begin{cor}
The admissible subcategory $\ca=\langle E_1, \ldots, E_{13} \rangle^\perp$ of $D^b(\Coh X)$ has vanishing Grothendieck group. In other words, $\ca \subseteq D^b(\Coh X)$ is a phantom subcategory.
\end{cor}

Next, we recall the following Theorem of Hille \& Perling \cite{HillePerling}.

\begin{thm} \label{T:HP}
Let $S$ be a smooth projective rational surface over an algebraically closed field. Then there is a tilting bundle on $S$.
\end{thm}

In combination with Theorem \ref{T:Krah}, this has the following consequence.
\begin{cor}\label{C:Main}
There exists an associative finite dimensional $\C$-algebra $A$ of finite global dimension, such that the bounded derived category $D^b(A)$ of finite dimensional $A$-modules admits an admissible subcategory $\cp$ with vanishing Grothendieck group $K_0(\cp)$.
\end{cor}
\begin{proof}
Let $X$ be the smooth projective rational surface in Theorem \ref{T:Krah}. By Theorem \ref{T:HP}, $X$ has a tilting bundle\footnote{Alternatively, as Krah pointed out to me while I was preparing this note, $X$ even has a full \emph{strong} exceptional collection and thus a tilting object.}  $\ct$. Thus tilting theory yields a triangle equivalence
\begin{align}\label{E:Teq}
\Phi:=\Hom_{D^b(X)}(\ct, -) \colon D^b(\Coh X) \xrightarrow{\cong}ÊD^b(\End_X(\ct)), 
\end{align}
where $\End_X(\ct)$ is a finite dimensional algebra, which has finite global dimension since $X$ is smooth.
Then $\cp=\Phi(\ca)$ has the desired properties by Theorem \ref{T:Krah}.
\end{proof}

\begin{rem}
By \cite{HillePerling}, the algebra $A$ is strongly quasi-hereditary, indeed they show that the standard modules $\Delta_i$ satisfy $\Ext^{>1}_A(\Delta_i, \Delta_j)=0$ -- in particular, $\prdim_A \Delta_i \leq 1$ for all $i$.   
\end{rem}

\section{Presilting objects without completion}
We first introduce the relevant definitions.

\begin{defn}
Let $\cd$ be a triangulated category. An object $P \in \cd$ is called \emph{presilting} if
\begin{align}
\Hom_\cd(P, P[>0])=0.
\end{align}
A presilting object $S$ is called \emph{silting} if additionally
\begin{align}
\thick(S) = \cd.
\end{align}
A presilting object $C$ is called \emph{partial silting}, if there exists an object $C' \in \cd$ such that $C \oplus C'$ is a silting object.
\end{defn}

Recently, examples of presilting objects that are not partial silting have been discovered for derived categories of certain finite dimensional algebras (with silting objects) via partially wrapped Fukaya categories, cf.~\cite{LZ} and also \cite{JSW}. The key ingredient is a triangulated subcategory $\cu$ that appears in all these examples. It does not have any presilting objects and $K_0(\cu) \cong \Z^2$. Our example Corollary \ref{C:Main2} of a presilting object that cannot be completed to a silting object, builds on the existence of the phantom category $\cp$ in Corollary \ref{C:Main} -- in particular, it has a quite different flavour. 

More precisely, the examples in \cite{LZ} and \cite{JSW} are gentle algebras. In the proposition below, we collect some properties of presilting objects for this class of algebras. In particular, statement (a)  
shows that for gentle algebras, there are no phantom categories that are perpendicular to presilting objects, in contrast to the example in Corollary \ref{C:Main2}.

\begin{prop}
Let $A$ be a finite dimensional gentle algebra and let $r=\rk K_0(D^b(A))$. Let $S$ in $D^b(A)$ be a presilting object. 
\begin{itemize}
\item[(a)] If $S$ has $r$ non-isomorphic indecomposable direct summands then $S$ is silting.(cf. \cite[Prop. 3.7]{APS}).
\item[(b)] If $S$ has $r-1$ non-isomorphic indecomposable direct summands then $S$ is partial silting (cf. \cite[Prop. 4.15]{JSW} and also \cite[Prop. 1.4.]{LZ} in a special case).
\end{itemize}
\end{prop}


\begin{ex}\label{Ex:ExcSeq}
Let $\cd$ be a proper\footnote{That is, $\dim_k \bigoplus_{i \in \Z} \Hom_\cd(X, Y[i])<\infty$ for all $X, Y$ in $\cd$. For example, this holds for $\cd=D^b(\Coh X)$ for a smooth projective variety $X$ over $\C$ and for $\cd=D^b(B)$ for a finite dimensional $\C$-algebra $B$ of finite global dimension.} $k$-linear triangulated category and assume that $E_1, \ldots, E_t \in \cd$ is an exceptional sequence, i.e.
\begin{itemize}
\item[(a)] $\Hom_\cd(E_i, E_j[\Z])=0$ for all $i>j$.
\item[(b)] $\Hom_\cd(E_i, E_i[\neq 0])=0$ and $\Hom_\cd(E_i, E_i)=\C$ for all $i$.
\end{itemize}
Then $\cd$ contains a presilting object. 

Indeed, since $\cd$ is proper, for every $i$ there is a maximal $s_i \in \Z$ with $\Hom_\cd(\bigoplus_j E_j, E_i[s_i]) \neq 0$. By property (a) and by taking appropriate shifts of the $E_i$, we can assume that all these $s_i\leq 0$. Then $\bigoplus_j E_j$ is silting in $\thick(E_1, \ldots, E_t)$ and presilting in $\cd$. It is a silting object in $\cd$ if and only if the exceptional sequence is \emph{full}, i.e. $\thick(E_1, \ldots, E_t)=\cd$.
\end{ex}

Applying this example to the exceptional sequence for the smooth surface $X$ in Theorem \ref{T:Krah} and the corresponding exceptional sequence for the finite dimensional $\C$-algebra $A$ in Corollary \ref{C:Main} yields presilting objects in $D^b(\Coh X)$ and $D^b(A)$ that are \emph{not} silting. We now observe that these presilting objects cannot be completed to silting objects, i.e. they are \emph{not} partial silting. 

We need the following result of Aihara \& Iyama \cite[Thm 2.27 \& Cor 2.28]{AiharaIyama}.
\begin{thm}\label{t:AI}
Let $\cd$ be a  Krull--Schmidt triangulated category with a silting object. Then $K_0(\cd)\cong \Z^r$ and all silting objects in $\cd$ have precisely $r$ non-isomorphic indecomposable direct summands.
\end{thm} 

In combination with Example \ref{Ex:ExcSeq}, we obtain the following. 

\begin{lem}\label{L:PreNoPart}
Let $\cd$ be a proper $k$-linear Krull--Schmidt triangulated category with a silting object. If there exists a non-full exceptional sequence $\mathbb{M}$ of length $\rk K_0(\cd)$, then there exists a presilting object that is \emph{not} partial silting.
\end{lem}
\begin{proof}
By Example \ref{Ex:ExcSeq}, every non-full exceptional sequence $\mathbb{E}$ in $D^b(A)$ yields a presilting object $P(\mathbb{E})$ that is \emph{not} silting. By construction, the number of non-isomorphic indecomposable direct summands of $P(\mathbb{E})$ equals the length of $\mathbb{E}$. In particular, the presilting object $P(\mathbb{M})$ has already $\rk K_0(\cd)$ non-isomorphic indecomposable direct summands. So there cannot be an object $C$ in $\cd$ such that $P(\mathbb{M}) \oplus C$ is silting by Theorem \ref{t:AI}.
\end{proof}

\begin{cor}\label{C:Main2}
Let $X$ and $A$ be as in Theorem \ref{T:Krah} and Corollary \ref{C:Main}, respectively. Then $D^b(A)$ and $D^b(\Coh X)$ have silting objects. Moreover, they contain presilting objects that are \emph{not} partial silting. 
\end{cor}
\begin{proof} We show that the assumptions in Lemma \ref{L:PreNoPart} are met. By the triangle equivalence $\Phi$ in \eqref{E:Teq} it is enough to give the argument for $D^b(\Coh X)$.
The $\C$-linear category $D^b(\Coh X)$ is Krull--Schmidt, since it is Hom-finite and idempotent complete.
Since $X$ is smooth, $D^b(\Coh X)$ is proper and the tilting object $\ct$ in Theorem \ref{T:HP} is silting. Finally, we can use the non-full exceptional sequence from Theorem \ref{T:Krah} as $\mathbb{M}$.
\end{proof}

\end{document}